\documentclass[a4paper,11pt]{article}

\usepackage{amsmath}
\usepackage{amssymb}
\usepackage{amsthm}
\usepackage{graphicx}
\usepackage{enumerate}
\usepackage[dvipsnames]{xcolor}
\usepackage{hyperref}

\newtheorem{theorem}{Theorem}[section]
\newtheorem{proposition}[theorem]{Proposition}

\newtheorem{corollary}[theorem]{Corollary}

\newtheorem{assumption}{Assumption A}

\theoremstyle{definition}

\begin{document}
	
\title{Spectral Projected  Subgradient Method for Nonsmooth Convex Optimization Problems}
\author{Nata\v sa Kreji\' c\footnote{Department of Mathematics and Informatics, Faculty of Sciences, University of Novi
Sad, Trg Dositeja Obradovi\' ca 4, 21000 Novi Sad, Serbia. e-mail: \texttt{natasak@uns.ac.rs}}, Nata\v sa Krklec Jerinki\' c,\footnote{Department of Mathematics and Informatics, Faculty of Sciences, University of Novi Sad, Trg Dositeja Obradovi\' ca 4, 21000 Novi Sad, Serbia. e-mail: \texttt{natasa.krklec@dmi.uns.ac.rs}} Tijana Ostoji\' c \footnote{Department of Fundamental Sciences, Faculty of Technical Sciences, University of Novi
Sad, Trg Dositeja Obradovi\' ca 6, 21000 Novi Sad, Serbia. e-mail: \texttt{tijana.ostojic@uns.ac.rs}} }
\date{August 8, 2022}
\maketitle

\begin{abstract}
We consider constrained optimization problems with a nonsmooth objective function in the form of mathematical expectation. The Sample Average Approximation (SAA) is used to estimate the objective function and variable sample size strategy is employed.  The proposed algorithm combines an SAA subgradient with the spectral coefficient in order to provide a suitable direction which improves the performance of the first order method as shown by numerical results. The step sizes are chosen from the predefined interval and the almost sure convergence of the method is proved under the standard assumptions in stochastic environment. To enhance the performance of the proposed algorithm, we further specify the choice of the step size by introducing an Armijo-like procedure  adapted to this framework. Considering the computational cost on machine learning problems, we conclude that the line search improves the performance significantly.   Numerical experiments conducted on finite sum problems also reveal that the variable sample strategy outperforms the full sample approach. 
\end{abstract}


\section{Introduction}
We consider the following constrained optimization problem 
\begin{equation} 
\min_{x\in\Omega} f(x)=E(F(x,\xi)),
\label{prob1}
\end{equation}
where $\Omega\subset\mathbb{R}^n$ is a convex, closed set, $F:\mathbb{R}^n \times \mathbb{R}^m \rightarrow \mathbb{R}$ is continuous and convex function with respect to $x$,   
bounded from below,  $\xi:\mathcal{A}\rightarrow \mathbb{R}^m$ is random vector and $(\mathcal{A}, \mathcal{F}, P)$  is a probability space. Notice that F is locally Lipschitz 
as a consequence of convexity, \cite{BKM14}, but possibly  nonsmooth. The importance of the stochastic optimization problem arising from various  scientific fields generated a large amount  of literature in the recent years. Due to difficulty in computing the mathematical expectation in general, the common approach is to approximate the original objective function $f(x)$ by applying  the Sample Average Approximation (SAA) function 

\begin{equation} 
f_{\mathcal{N}}(x)=\frac{1}{N}\sum_{i\in\mathcal{N}}f_i(x),
\label{SAA}
\end{equation}
where $ f_i(x)=F(x,\xi_i) $ and $ N=|\mathcal{N}| $ determines the size of a sample used for approximation. The sample vectors  $\xi_i, i\in\mathcal{N}$ are assumed to be independent and identically distributed (i.i.d.). 

The sample size $N$ determines the  precision of the approximation, but it also influences the computational  cost. In order to achieve the a.s. convergence, one needs to push the sample size to infinity in general. Even if the original problem is already in the SAA form, i.e., if we are dealing with finite sum problems, the costs of employing the full sample at each iteration can be large and thus the variable sample size (VSS) strategy is often applied.  Finding a good way of  varying  the sample size as well as choosing a sample is a problem itself and it has been the subject of many research efforts  (e.g. \cite{BCT}, \cite{SBNKNKJ}, \cite{HT2}, \cite{NKNKJ}). Although this may affect the algorithm a lot, a suitable sample size strategy will not be the main concern of this paper. In order to prove a.s. convergence we assume here that the sample size tends to infinity and leave the problem of determining the optimal VSS strategy in this framework for future work.

In this paper we propose a framework for solving nonsmooth constrained optimization problem \eqref{prob1} assuming that the feasible set $\Omega$ is easy to project on (for example a box or ball in $\mathbb{R}^n$). This allows us to apply a method of the Spectral Projected Gradient type.
The Spectral Projected Gradient (SPG) method, originally proposed in \cite{Birgin1}, is  well known for its efficiency and simplicity and it has been widely used and developed as a solver of constrained optimization problems \cite{greta}, \cite{GRMS}, \cite{NKNKJ}, \cite{tan}. The step length selection strategy in SPG method is crucial for faster convergence with respect to classical gradient projection methods because it involves second-order information related to the spectrum of the Hessian matrix.
SPG methods for finite sums problem have been investigated in \cite{greta}, \cite{tan}. In \cite{tan} they are used in combination with the stochastic gradient method and the convergence is proved assuming that the  full gradient is calculated in  every $m$ iterations.  In \cite{greta} the subsampled spectral gradient methods are analyzed and the effect of the choice of the spectral coefficient is investigated.
In \cite{GRMS} the SPG direction is employed within Inexact Restoration framework to address nonlinear optimization problems with nonconvex constraints. The SPG methods for problems with continuously differentiable objective function given in form of mathematical expectation have been analyzed in \cite{NKNKJ}.

Another important class of  nonsmooth optimization methods are bundle methods, especially if accuracy in the solution and reliability are a concern (see \cite{surveymm}, \cite{birdeye} and the references therein). The main idea of bundle methods is to make an approximation of the whole subdifferential of the objective function instead of using only one arbitrary subgradient at each point. The advantage of these methods compared to the classical subgradient methods is that they use more information about the local behavior of the function by approximating the subdifferential of the objective function with subgradients from previous iterations that form a bundle. A great number of bundle methods in combination with Newton-type methods \cite{LV98} and the trust-region method \cite{ANR16} have been developed. The modification of bundle methods has been used for nonconvex problems \cite{Mif82}, constrained problems \cite{SS05} and multi-criteria problems \cite{Mie98}, as well. The drawback of these methods is requirement of solving at least one quadratic programming subproblem in each iteration, which can be time-consuming, especially for large scale problems. Proximal bundle methods (see for example \cite{LS97}, \cite{misa} \cite{MN92}) are based on the bundle methodology and have the ability to provide exact solutions even if most of the time the available information is inaccurate, unlike their forerunner variants. 
The proximal bundle method takes ideas from subgradient method and proximal method \cite{senior}, \cite{birdeye}. First one can be seen as extension of gradient methods in smooth optimization and the second one is a variant of proximal point method, which minimizes the original function plus a quadratic part.  

The method proposed in this paper  is a subgradient method but differs from the existing ones in the literature in several ways. We propose a way to plug VSS-SPG ideas into the nonsmooth framework. Since the objective function may be nonsmooth, we have to use subgradients instead of gradients and thus we refer to the core algorithm as SPS - Spectral Projected Subgradient method. The spectral coefficient is calculated by employing consecutive subgradients of possibly different SAA functions and the safeguard which provides positive, bounded spectral coefficients is used. We prove a.s. convergence under the standard assumptions for stochastic environment.  Moreover, in order to improve the performance of the algorithm, we also propose a line search variant of SPS named LS-SPS.  By specifying the line search technique we ensure that LS-SPS falls into the SPS framework and thus the same convergence results hold. Although the descent property of the search direction is desirable, it is not necessary in each iteration to ensure the convergence result. The proposed line search is well defined and the a.s. convergence is achieved even if the search direction is not a descent one for the SAA function. 

Although the proposed algorithms are constructed to cope with unbounded sample sizes, they can also be applied on finite sum problems and we devote a part of considerations to this important special class as well. Preliminary numerical results on machine learning problems modeled by Hinge Loss functions reveal several advantages of the proposed method: 
a) spectral method outperforms the plain subgradient method; 
b) VSS method reduces the costs of the full sample method when applicable; 
c) line search improves the  behavior of the SPS method with predefined step size sequence. 

Therefore, the main contributions of this work are the following: 
\begin{itemize}
\item[i)] The stochastic SPG method is adapted to the  nonsmooth framework; 
\item[ii)] The a.s. convergence of the proposed SPS method is proved under the standard assumptions; 
\item[iii)] The SPS is further upgraded by introducing  a specific line search technique resulting in LS-SPS; 
\item[iv)] Numerical results on machine learning problems show the efficiency of the proposed method, especially LS-SPS. 
\end{itemize}

This paper is structured as follows. Details of the proposed algorithm are presented in the next section. Section 3 deals with the convergence of the resulting algorithm and also introduces a line search algorithm. In Section 4 we provide some implementation details and discuss results of the relevant numerical experiments.  We conclude in Section 5.

\section{The Method}

In this section, we describe the subsampled spectral projected subgradient framework algorithm for nonsmooth probems  - SPS. For any  given $z \in \mathbb{R}^n$,  we denote by $P_{\Omega}(z)$ the orthogonal projection of $z$ onto $\Omega.$ Throughout the paper $\mathcal{N}_k$ denotes the sample used to approximate the objective function and $N_k$ denotes its cardinality.

\noindent {\bf Algorithm 1: SPS} \\({\bf S}pectral {\bf P}rojected {\bf S}ubgradient Method for Nonsmooth Optimization)
\label{SPGNS}
\begin{itemize}
\item[S0] \textit{Initialization.} Given $N_0 \in \mathbb{N},$ $x_0\in\Omega,$ $0<C_1 <1< C_2< \infty, $ $0 <  \underline{\zeta}\leq  \overline{\zeta}< \infty,$ $\zeta_0 \in \left[\underline{\zeta},\overline{\zeta}\right].$ Set k = 0.

\item[S1] \textit{Direction.} Choose  $\bar{g}_{k}\in  \partial f_{\mathcal{N}_{k}}(x_k)$ and set $p_{k}=-\zeta_k \bar{g}_{k}$. 

\item[S2] \textit{Step size.} Choose $\alpha_k \in \left[ C_1/k,C_2/k\right]$.

\item[S3] \textit{Main update.} Set $x_{k+1} = P_{\Omega}(x_k+\alpha_k  p_{k})$ and  $s_k = x_{k+1}-x_k$.

\item[S4] \textit{Sample size update.} Chose $ N_{k+1} \in \mathbb{N} $.

\item[S5] \textit{Spectral coefficient update.} Calculate $y_k = g_{\mathcal{N}_{k}}(x_{k+1})- \bar{g}_{k}$ where $g_{\mathcal{N}_k}(x_{k+1}) \in \partial f_{\mathcal{N}_{k}}(x_{k+1})$.
Set 
$\zeta_{k+1} = \min \lbrace\overline{\zeta},\max\lbrace\underline{\zeta},\frac{s_k^T s_k}{s_k^T y_k}\rbrace\rbrace.$  
 
\item[S6] Set $k:=k+1$ and go to S1. 
\end{itemize}

Let us now comment the algorithm above.  In Step S1 we calculate the direction by choosing a subgradient of the SAA function $f_{\mathcal{N}_k}$ and taking the opposite direction multiplied by the spectral coefficient. Notice that the safeguard in Step S5 ensures that the negative subgradient direction is retained. However, this direction does not have to be descent for the function $f_{\mathcal{N}_k}$ since we take an arbitrary subgradient. 

Within Step S2 we choose the step size $\alpha_k$ from the given interval. The constants $C_1$ and $C_2$ can be arbitrary small and large, respectively, allowing a wide range of feasible step sizes. This choice was motivated by the common assumption on the step size sequence in stochastic algorithms:
\begin{equation}
 \sum_{k=1}^{\infty}\alpha_k^2 <\infty, \quad \sum_{k=1}^{\infty}\alpha_k=\infty.
\label{ss_rule}
\end{equation}
Notice that the choice in S2 ensures that the sequence of step sizes of SPS algorithm satisfies \eqref{ss_rule}. 
After finding the direction and the step size, we project the point $x_{k} +\alpha_k p_k$ onto the feasible set $\Omega$ and thus we retain feasibility in all the iterations of the algorithm. 

In Step S4 we chose the sample size to be used in the subsequent iteration. In order to prove convergence result, we will assume that $N_k$ tends to infinity or achieves and retains at the maximal sample size in the case of finite sums. So, the simplest way to ensure this is to increase the sample size at each iteration. However, we state this step in the most generic way to emphasize that other choices are feasible as well, including some adaptive strategies. 

In Step S5 $y_k$ is calculated as a difference of two subgradients of the same approximate function $f_{\mathcal{N}_k}$, but different approaches are feasible as well. For instance, one can use subgradients of different functions $y_k = g_{\mathcal{N}_{k+1}}(x_{k+1})- g_{\mathcal{N}_k}(x_k)$. This can reduce the costs, especially if the sample is not cumulative, but also brings additional  noise into the spectral coefficient since the subgradients are calculated for two different functions in general.  However, if we have a finite sum problem  and the full sample is reached, the cost of calculating the subgradient may be reduced to one subgradient per iteration since one can obviously take $\bar{g}_{k+1}=g_{\mathcal{N}_{k}}(x_{k+1})$. In general, another choice could be  $y_k = g_{\mathcal{N}_{k+1}}(x_{k+1})- g_{\mathcal{N}_{k+1}}(x_k).$ This reduces the influence of a noise and usually provides  better approximation of the spectral coefficient of the true objective function, but it requires additional evaluations. Although the choice of $y_k$ was addressed in the literature (see \cite{greta} for example), in general  it remains an open question which requires thorough analysis before drawing the final conclusions.
It is important to point out that the choice of $y_k$ does not affect the convergence analysis and the theoretical results obtained in next section, but it may affect the algorithm's performance significantly.

\section{Convergence}

In this section we analyze conditions needed for a.s.  convergence of the SPS algorithm. A standard assumptions for stochastic environment is stated below.  Recall that samples are assumed to be i.i.d.

\begin{assumption}
Assume that $f_i(x)=F(x,\xi_i),$ $i=1,2,\ldots,$ are continuous, convex and bounded from below with a constant $C.$  Moreover, assume that the function $F$ is dominated by a P-integrable function on any compact subset of $\mathbb{R}^n$.
\label{pp_Lip}
\end{assumption}

Notice that Assumption A\ref{pp_Lip} implies that $ f $ is convex and continuous function as well as $ f_\mathcal{N} $ for any given $N$. 
Moreover, the Uniform Law of Large Numbers (ULLN) implies  that $f_\mathcal{N}(x)$ a.s converges uniformly  to $f(x)$ on any compact subset  $S \subseteq \mathbb{R}^n$ (see Theorem 7.48 in \cite{shapiro} for instance), i.e., 
\begin{equation}
 \lim_{N \rightarrow \infty}\sup_{x \in S}|f_{N}(x)-f(x)|=0 \quad \mbox{a.s.}   
 \label{Shapiro-ULLN}
 \end{equation}
Also notice that \eqref{Shapiro-ULLN} holds trivially if the sample is finite and the full sample is eventually achieved and retained.

The main result, a.s. convergence of Algorithm SPS, is stated in the following theorem. We assume that the feasible set is compact, although this assumption may be relaxed as we will show in the sequel. Moreover, recall that the convexity implies that the functions $f_i$ are locally Lipschitz continuous  and thus for every $x$ and $i$ there exists $L_i(x)$ such that for all $g \in \partial f_{i}(x) $ there holds $\|g\| \leq L_i(x)$. However, since there can be infinitely many functions $f_i$ in general we assume that the chosen subgradients are uniformly bounded. This may be accomplished by scaling the subgradient by its norm for instance. 
Let $X^*$ and $f^*$ be the set of solutions and the optimal value of problem  \eqref{prob1}, respectively.  
The convergence result is as follows. 

\begin{theorem}
Suppose that Assumption A\ref{pp_Lip} holds and $\lbrace x_k \rbrace$ is a sequence generated by Algorithm SPS where $N_k\to\infty$.  Assume also that $\Omega$ is  compact and convex and there exists $G$ such that $|| \bar{g}_k ||\leq G$ for all $k$. Then
\begin{equation}
\label{t1a}
\liminf_{k \to \infty} f(x_k)=f^* \quad \text{a.s.}.
\end{equation}
 Moreover,  
 \begin{equation}
\label{t1b}
\lim_{k\rightarrow\infty} x_k=x^*  \quad \text{a.s.}
\end{equation}
 for some $x^* \in X^*,$ provided that $\sum_{k=0}^{\infty}  e_k /k <\infty$, where $e_k:=\max_{x\in \Omega}|f_{\mathcal{N}_k}(x)- f(x)|.$

\label{Th1}
\end{theorem}

\begin{proof}

Denote by $\mathcal{W}$ the set of all possible sample paths of SPS algorithm. Suppose that \eqref{t1a}  does not hold, i.e.,  $\liminf_{k \to \infty} f(x_k)=f^* $ does not happen with probability 1. In that case there exists a subset  of sample paths $\tilde{\mathcal{W}}\subseteq \mathcal{W}$ such that $ P(\tilde{\mathcal{W}})>0$ and  
for every $w \in \tilde{\mathcal{W}} $ there holds 
$$\liminf_{k\rightarrow \infty} f(x_k(w))> f^*,$$ 
i.e., there exists $\varepsilon(w) >0$ small enough  such that  
$f(x_k(w))-f^* \geq 2\varepsilon(w) $ for all $k$. 
Since $f$ is continuous on the feasible set $\Omega$,   there exists $\tilde{y}(w)\in \Omega$ such that
$f(\tilde{y}(w))=f^*+\varepsilon(w).$
This further implies 
$$f(x_k(w))-f(\tilde{y}(w))=f(x_k(w))-f^*-\varepsilon(w) \geq 2\varepsilon(w)-\varepsilon(w)=\varepsilon(w).$$
Let us take an arbitrary $w \in \tilde{\mathcal{W}}.$ Denote $z_{k+1}(w):=x_k(w)+\alpha_k(w) p_k(w)$. 
Notice that nonexpansivity of orthogonal projection and the fact that $\tilde{y} \in \Omega$ together imply  
\begin{equation}
||x_{k+1}(w)-\tilde{y}(w)||=||P_{\Omega}(z_{k+1}(w))-P_{\Omega}(\tilde{y}(w))||\leq||z_{k+1}(w)-\tilde{y}(w)||.
\label{nej_proj}
\end{equation}
Furthermore, using the fact that $ \bar{g}_{k} $  is subgradient of the convex function $ f_{\mathcal{N}_{k}} $,  $ \bar{g}_{k}\in  \partial f_{\mathcal{N}_{k}}(x_k) $, we have $f_{\mathcal{N}_k}(x_k)-f_{\mathcal{N}_k}(\tilde{y})\leq \overline{g}^T_k(x_k-\tilde{y})$ and dropping the $w$ in order to facilitate the reading we obtain 
\begin{eqnarray}
||z_{k+1}- \tilde{y}||^2 
&=&\nonumber
||x_{k}+\alpha_k p_k-\tilde{y}||^2 = ||x_{k}-\alpha_k \zeta_k\overline{g}_k-\tilde{y}||^2 \\ \nonumber
&=& 
||x_{k}-\tilde{y}||^2-2\alpha_k \zeta_k\overline{g}^T_k \left(x_{k}-\tilde{y}\right) + \alpha_k^2 \zeta^2_k || \overline{g}_{k}||^2 \\ \nonumber
&\leq & 
||x_{k}-\tilde{y}||^2+2\alpha_k \zeta_k (f_{\mathcal{N}_{k}}(\tilde{y})-f_{\mathcal{N}_{k}}(x_k))+ \alpha_k^2 \zeta_k^2||\overline{g}_{k}||^2 \\ \nonumber
&\leq &
||x_{k}-\tilde{y}||^2 +2\alpha_k \zeta_k(f(\tilde{y})-f(x_k)+2e_k) +\alpha_k^2 \zeta_k^2||\overline{g}_{k}||^2\\ \nonumber
&\leq&
||x_{k}-\tilde{y}||^2-2\alpha_k\zeta_k(f(x_k)-f(\tilde{y}))+4e_k\alpha_k \overline{\zeta}+\alpha_k^2\overline{\zeta}^2G^2\\ \nonumber
&\leq& 
||x_{k}-\tilde{y}||^2-2 \alpha_k\underline{\zeta}\varepsilon+4e_k\alpha_k\overline{\zeta}+\alpha_k^2\overline{\zeta}^2G^2\\ 
&=&
||x_{k}-\tilde{y}||^2-\alpha_k \left(2 \underline{\zeta}\varepsilon-4e_k\overline{\zeta} -\alpha_k\overline{\zeta}^2G^2\right).
\label{dod1}
\end{eqnarray}

By ULLN we have  
$ \lim_{k \rightarrow \infty} e_k =0  \text{ a.s.}$, or more precisely, $ \lim_{k \rightarrow \infty} e_k(w) =0  $ for almost every $w \in \mathcal{W}$. Since $P(\tilde{\mathcal{W}})>0$, there must exist a sample path  $\tilde{w}\in \tilde{\mathcal{W}}$ such that $$\lim_{ k\rightarrow \infty} e_k(\tilde{w}) =0.$$ 
This further implies the existence of  $\tilde{k}(\tilde{w})\in \mathbb{N}$ such that for all $k\geq \tilde{k}(\tilde{w})$ we have 
\begin{equation}
\alpha_k (\tilde{w}) \overline{\zeta}^2G^2+4e_k (\tilde{w}) \overline{\zeta}\leq  \varepsilon(\tilde{w}) \underline{\zeta}
\label{nejednakost}
\end{equation}
because step S2 of SPS algorithm implies that $\lim_{k \to \infty } \alpha_k=0$ for any sample path.
Furthermore, since \eqref{dod1} holds for all $w \in \tilde{\mathcal{W}}$ and thus for $\tilde{w}$ as well, from \eqref{nej_proj}-\eqref{nejednakost} we obtain
 
\begin{eqnarray*}
||x_{k+1}(\tilde{w})- \tilde{y}(\tilde{w})||^2 
&\leq & 
||x_{k}(\tilde{w})-\tilde{y}(\tilde{w})||^2\\\nonumber
&-& \alpha_k(\tilde{w}) \left(2 \underline{\zeta}\varepsilon(\tilde{w})-4e_k(\tilde{w})\overline{\zeta} -\alpha_k(\tilde{w})\overline{\zeta}^2G^2\right) \\ \nonumber
& = & 
||x_{k}(\tilde{w})-\tilde{y}(\tilde{w})||^2\\\nonumber 
&+& \alpha_k(\tilde{w}) \left(4e_k(\tilde{w})\overline{\zeta} + \alpha_k(\tilde{w})\overline{\zeta}^2G^2-2 \underline{\zeta}\varepsilon(\tilde{w})\right) \\\nonumber
&\leq  & 
||x_{k}(\tilde{w})-\tilde{y}(\tilde{w})||^2+\alpha_k (\tilde{w})\left(\underline{\zeta}\varepsilon(\tilde{w})-2 \underline{\zeta}\varepsilon(\tilde{w})\right) \\ \nonumber
&=& 
||x_{k}(\tilde{w})-\tilde{y}(\tilde{w})||^2-\alpha_k (\tilde{w})\underline{\zeta}\varepsilon(\tilde{w}). \\ \nonumber
\end{eqnarray*}
and
$$||x_{k+s}(\tilde{w})-\tilde{y}(\tilde{w})||^2\leq ||x_{k}(\tilde{w})-\tilde{y}(\tilde{w})||^2- \varepsilon (\tilde{w})\underline{\zeta} \sum_{j=0}^{s-1}\alpha_j(\tilde{w}).$$ 
Letting $s \to \infty$ yields a contradiction since $\sum_{k=0}^{\infty} \alpha_k \geq \sum_{k=0}^{\infty} C_1 /k =\infty$ for any sample path and  we conclude that \eqref{t1a}  holds.

Now, let us prove \eqref{t1b} under the additional assumption $\sum_{k=0}^{\infty}  e_k /k <\infty$. Notice that this assumption implies that  $\sum_{k=0}^{\infty} \alpha_k e_k <\infty$ since $\alpha_k \leq C_2 /k$.  Since \eqref{t1a} holds, we know that 
\begin{equation}
\liminf_{k\rightarrow\infty} f(x_k(w))=f^*,
\label{liminf}
\end{equation}
for almost every $w \in \mathcal{W}$. In other words, there exists $\overline{\mathcal{W}} \subseteq \mathcal{W}$ such that $P(\overline{\mathcal{W}})=1$ and \eqref{liminf} holds for all $w \in \overline{\mathcal{W}}$. Let us consider arbitrary $w \in \overline{\mathcal{W}}$. We will show that $\lim_{k\rightarrow\infty} x_k(w)=x^*(w) \in X^*$ which will imply the result \eqref{t1b}.  Once again let us  drop $w$ to facilitate the notation. 

Let $K_1\subseteq\mathbb{N}$ be a subsequence of iterations such that
$$\lim_{k\in K_1}f(x_k)=f^*.$$
Since $\lbrace{ x_k \rbrace}_{k\in K_1}\subseteq \lbrace{ x_k \rbrace}_{k\in\mathbb{N}}$ and $\lbrace{ x_k \rbrace}_{k\in\mathbb{N}}$ is bounded because of feasibility and the compactness of the feasible set $\Omega$,  there follows that $\lbrace{ x_k \rbrace}_{k\in K_1}$ is also bounded and  there exist $K_2\subseteq K_1$ and $\tilde{x}$ such that
$$\lim_{k\in K_2} x_k=\tilde{x}.$$
Then, we have
$$f^*=\lim_{k\in K_1}f(x_k)=\lim_{k\in K_2}f(x_k)=f(\lim_{k\in K_2} x_k)=f(\tilde{x}).$$
Therefore, $f(\tilde{x})=f^*$ and we have $\tilde{x}\in X^*$. Now, we show that the whole sequence of iterates converges. 
Let $\lbrace{ x_k \rbrace}_{k\in K_2}:=\lbrace{ x_{k_i} \rbrace}_{i\in\mathbb{N}}.$ Following the steps of  \eqref{dod1} and using the fact that $f(x_k)\geq f(\tilde{x})$ for all $k$, we obtain that the following holds for any $s\in\mathbb{N}$

\begin{eqnarray} 
||x_{k_i+s}- \tilde{x}||^2  
&\leq &
||x_{k_i}-\tilde{x}||^2+4\overline{\zeta}\sum_{j=0}^{s-1}e_{k_i+j}\alpha_{k_i+j}+\overline{\zeta}^2G^2\sum_{j=0}^{s-1}\alpha_{k_i+j}^2\\
&\leq &
||x_{k_i}-\tilde{x}||^2+4\overline{\zeta}\sum_{j=0}^{\infty}e_{k_i+j}\alpha_{k_i+j}+\overline{\zeta}^2G^2\sum_{j=0}^{\infty}\alpha_{k_i+j}^2\\
&=&
||x_{k_i}-\tilde{x}||^2+4\overline{\zeta}\sum_{j=k_i}^{\infty}e_j\alpha_j+\overline{\zeta}^2G^2\sum_{j=k_i}^{\infty}\alpha_j^2.
\end{eqnarray}

Due to the fact that  $\sum_{j=k_i}^{\infty}e_j\alpha_j$ and $\sum_{j=k_i}^{\infty}\alpha_j^2$ are the residuals of convergent sums,  for any  $s\in \mathbb{N} $ there holds
$$||x_{k_i+s}- \tilde{x}||^2  \leq a_i, \text{ where } \lim_{i\to\infty}a_i=0.$$
Thus, for all $s\in \mathbb{N} $ we have
$$\limsup_{i\to\infty}||x_{k_i+s}- \tilde{x}||^2\leq \limsup_{i\to\infty} a_i=\lim_{i\to\infty}a_i=0.$$
Since $s\in \mathbb{N} $ is arbitrary, there follows
$\limsup_{k\to\infty}||x_k-\tilde{x}||^2=0,$
i.e. $\lim_{k\to\infty}x_k=\tilde{x}$ which completes the proof.
%
%
%
%
%
%
\end{proof}

Let us comment on $e_k$ first. We obtain \eqref{t1a} provided that the sample size tends to infinity in an arbitrary manner. The stronger result is achieved under assumption of fast enough increase of the sample size. Having in mind the interval for step size $\alpha_k$, we conclude that the assumption of summability needed for \eqref{t1b} is satisfied if $e_k\leq C_3 k^{-\nu}$ holds for all $k$ large enough and arbitrary $C_3>0$, where $\nu>0$ can be arbitrary  small.  While in general this can be hard to guarantee, for some classes of functions $F$ (e.g. function plus noise with finite
variance), the error bound for cummulative samples derived in \cite{HT2} yields $e_k\leq C \sqrt{\frac{\ln \ln N_k}{N_k}},$ for all $k$ large enough and some positive constant $C$ directly dependent on the noise variance. In that case, it can be shown that the simple choice of $N_k=k$ provides the sufficient growth needed for \eqref{t1b}. 

An important class of the problems that we consider is a finite sum problem 
\begin{equation}\label{fs}
\min_{x \in \Omega} f_\mathcal{N}(x)=\frac{1}{N}\sum_{i=1}^{N} f_i(x),
\end{equation}
where the functions $f_i$ are continuous, convex and  possibly nonsmooth. In that case, we do not need a dominance assumption since \eqref{Shapiro-ULLN} is trivially satisfied if the full sample size is eventually achieved and retained. Thus, $e_k=0$ for all $k$ large enough and  $\sum_{k=0}^{\infty} e_k/k <\infty$ trivially holds. Moreover, the compactness of the feasible set and convexity of $f$ imply the Lipschitz continuouity of each $f_i$ on $\Omega$ and thus the subgradients $\bar{g}_k$ are uniformly bounded. Furthermore, we also know that the functions are uniformly bounded from below on $\Omega$. Although the sample paths may differ, the convergence result is deterministic since the original objective function $f_N$ is eventually used. We summarize the result in the following corollary of the previous theorem.   

\begin{corollary}
Suppose that the functions $f_i, i=1,...,N$ are continuous and convex and  $\lbrace x_k \rbrace$ is a sequence generated by Algorithm SPS applied to \eqref{fs}.  Suppose that $N_k=N$ for all $k$ large enough and  $\Omega$ is  compact and convex. Then $
\lim_{k\rightarrow\infty} x_k=x^* \in X^*$. 
\label{Th2}
\end{corollary}

Since the compactness of $\Omega$ excludes the important class of constrained problems such as $x\geq 0$ which appear as subproblems in many cases, for instance in penalty methods, it is important to comment on the alternatives. The assumption of bounded $\Omega$ may be replaced with the assumption of bounded iterate sequence $\{x_k\}_k$ which is common in stochastic analysis. We state the result for completeness. 

\begin{theorem}
Suppose that Assumption A\ref{pp_Lip} holds and $\lbrace x_k \rbrace \subseteq \bar{\Omega}$ is a sequence generated by Algorithm SPS where $N_k \rightarrow \infty$ and $\bar{\Omega}\subseteq\Omega$ is bounded.  Assume  that $\Omega$ is  closed and convex and there exists $G$ such that $|| \bar{g}_k ||\leq G$ for all $k$. Then
$
\liminf_{k \to \infty} f(x_k)=f^*  \text{a.s.}.
$
 Moreover,  
 $\lim_{k\rightarrow\infty} x_k=x^* \in X^* $  a.s. 
 provided that $\sum_{k=0}^{\infty}  e_k/k <\infty$, where $e_k:=\max_{x\in \bar{\Omega}}|f_{\mathcal{N}_k}(x)- f(x)|.$

\label{Th12}
\end{theorem}

Now, let us see under which conditions we obtain the boundedness of iterates. Define an SAA error sequence as 
\begin{equation}
\label{ebar}
\bar{e}_k=|f_{\mathcal{N}_k}(x_k)-f(x_k)|+|f_{\mathcal{N}_k}(x^*)-f(x^*)|,
\end{equation} 
where $x^*\in X^*$ is an arbitrary solution point. We have the following result. 

\begin{proposition}
Suppose that Assumption A\ref{pp_Lip} holds and $\lbrace x_k \rbrace$ is a sequence generated by Algorithm SPS where $N_k\to\infty$.  Assume  that $\Omega$ is  closed and convex and there exists $G$ such that $|| \bar{g}_k ||\leq G$ for all $k$. Then there exists a compact set $\bar{\Omega}\subseteq\Omega$ such that $\lbrace x_k \rbrace \subseteq \bar{\Omega}$ provided that 
$\sum_{k=0}^{\infty} \bar{e}_k/k \leq C_4 <\infty.$ 
\label{Th123}
\end{proposition}

\begin{proof} Let $x^*$ be an arbitrary solution of the problem \eqref{prob1}. Then, by following similar steps as in \eqref{dod1} and using the nonexpansivity of the projection  \eqref{nej_proj} we obtain that the following holds for an arbitrary $k$
\begin{eqnarray}
\label{dod2}
||x_{k+1}-x^*||^2 
&\leq &\nonumber
||z_{k+1}-x^*||^2 = ||x_{k}-\alpha_k \zeta_k\overline{g}_k-x^*||^2 \\ \nonumber
&=& 
||x_{k}-x^*||^2-2\alpha_k \zeta_k\overline{g}^T_k \left(x_{k}-x^*\right) + \alpha_k^2 \zeta^2_k || \overline{g}_{k}||^2 \\ \nonumber
&\leq & 
||x_{k}-x^*||^2+2\alpha_k \zeta_k (f_{\mathcal{N}_{k}}(x^*)-f_{\mathcal{N}_{k}}(x_k))+ \alpha_k^2 \zeta_k^2||\overline{g}_{k}||^2 \\ \nonumber
&\leq &
||x_{k}-x^*||^2 +2\alpha_k \zeta_k(f(x^*)-f(x_k)+\bar{e}_k) +\alpha_k^2 \zeta_k^2||\overline{g}_{k}||^2\\ \nonumber
&\leq&
||x_{k}-x^*||^2+2\bar{e}_k \overline{\zeta}C_2/k+\alpha_k^2\overline{\zeta}^2 G^2\\ \nonumber
&\leq& 
||x_{0}-x^*||^2+2 C_2 \overline{\zeta} \sum_{k=0}^{\infty} \frac{\bar{e}_k}{k}  + \overline{\zeta}^2 G^2 \sum_{k=0}^{\infty} \frac{C^2_2}{k^2}.\nonumber
\end{eqnarray}
Thus, there exists a constant $ C_5$ such that $||x_{k}-x^*|| \leq C_5$ for all $k$, which completes the proof. 
\end{proof}
We summarize the convergence result for unbounded feasible set in the following theorem. 

\begin{theorem}
Suppose that Assumption A\ref{pp_Lip} holds, the feasible set  $\Omega $ is convex and  closed and $\lbrace x_k \rbrace$ is a sequence generated by Algorithm SPS where $N_k$ tends to infinity fast enough to provide  $\sum_{k=0}^{\infty} \bar{e}_k/k \leq C_4 <\infty$ with $\bar{e}_k$ given by \eqref{ebar}. Then
$$
\liminf_{k \to \infty} f(x_k)=f^*  \text{a.s.}.
$$
 Moreover,  
 $$\lim_{k\rightarrow\infty} x_k=x^*   \text{a.s.}
$$
 for some $x^* \in X^*$  provided that $\sum_{k=0}^{\infty}  e_k/k <\infty$, where $e_k:=\max_{x\in \bar{\Omega}}|f_{\mathcal{N}_k}(x)- f(x)|$ and $\bar{\Omega}$ is a compact set containing $\lbrace x_k \rbrace$. 

\label{Th1234}
\end{theorem}

\subsection{Improving the efficiency - Line Search SPS} 

Notice that SPS algorithm works with an arbitrary subgradient direction related to the current SAA function. However, in some applications such as Hinge Loss binary clasification, it is possible to provide a descent direction with respect to the SAA function \cite{nasprvi} by applying the procedure proposed in \cite{kineski} or gradient subsampling technique \cite{burke} for instance. On the other hand, it is well known that applying the  line search may improve the performance of the algorithm significantly, even in the stochastic environment. In order to make the SPS algorithm more efficient, we propose a line search technique adapted to the nonsmooth variable sample size framework to fit the SPS convergence analysis. The proposed line search does not require a descent search direction in order to be well defined, nor the convergence analysis depends on the descent property. So, the following property \eqref{descent} of $\bar{g_k}  \in \partial f_{\mathcal{N}_k} (x_k)$ is desirable, but not necessary in order to prove the convergence of the Line Search SPS (LS-SPS) algorithm presented in the sequel,
\begin{equation}
\label{descent}
 \sup_{g \in \partial f_{\mathcal{N}_k} (x_k)} g^T p_{k} \leq -\frac{m}{2}\|p_k\|^2 \quad \mbox{for some} \; m>0.
\end{equation}
{\bf{The LS procedure.}} Since we employ the spectral subgradient method, we use nonmonotone Armijo-type line search condition 
\begin{equation}
f_{\mathcal{N}_{k}}(x_k+\alpha_k  p_{k})\leq \max_{i \in [\max \{1,k-c\}, k]}{f_{\mathcal{N}_{i}}(x_i)}-\eta \alpha_k ||p_{k}||^2,
\label{Armijo}
\end{equation}
 where $p_{k}$ is the search direction as in  Step S1 of Algorithm 1.
The candidates for  $ \alpha_k$ that we consider are: $d_k$ and $ (d_k+1/k)/2$, where $d_k= \min \{1, C_2/k\}$. The reasoning behind this is the following. The choice of $\alpha_k=1/k$ is a typical choice that is suitable for obtaining a.s. convergence. The line search is employed to estimate if the  larger value of $\alpha_k$ may be used. Since the backtracking techniques usually start with 1, we take the minimum of 1 and $C_2/k$ as the initial choice. Although $C_2/k$ must be included to ensure the theoretical requirements of Step S2, one can take $C_2$ arbitrary large such that $d_k=1$ even for the large values of $k$. Thus, practically, 1 would be the initial choice in all practical applications. We set the middle of the interval $[1/k, d_k]$ as the second possible choice for step size in line search. Although other strategies are feasible as well, we reduce to these two guesses to avoid the computational costs of unsuccessful line search attempts. Thus, if $\alpha_k=d_k$ satisfies \eqref{Armijo}, we take this as a step size. If not, we check  \eqref{Armijo} with the medium value $\alpha_k=(d_k+1/k)/2$. If the condition is satisfied, we retain this choice, otherwise we set $\alpha_k=1/k$.

\noindent {\bf Algorithm 1: LS-SPS} \\({\bf L}ine {\bf S}earch {\bf S}pectral {\bf P}rojected {\bf S}ubgradient Method for Nonsmooth Optimization)
\label{SPGNSLS}
\begin{itemize}
\item[S0] \textit{Initialization.} Given $N_0 \in \mathbb{N},$ $x_0\in\Omega,$ $0<C_1<1 < C_2< \infty, $ $0 < \underline{\zeta}\leq  \overline{\zeta}< \infty,$ $\zeta_0 \in \left[\underline{\zeta},\overline{\zeta}\right], \eta \in (0,1), c \in \mathbb{N}.$ Set k = 0.

\item[S1] \textit{Direction.} Choose  $\bar{g}_{k}\in  \partial f_{\mathcal{N}_{k}}(x_k)$ satisfying \eqref{descent} if possible and set $p_{k}=-\zeta_k \bar{g}_{k}$. 

\item[S2] \textit{Step size.} Choose $\alpha_k \in \{d_k, (d_k+1/k)/2\}$ such that \eqref{Armijo} holds if possible. Otherwise set $\alpha_k=\frac{1}{k}.$

\item[S3] \textit{Main update.} Set $x_{k+1} = P_{\Omega}(x_k+\alpha_k  p_{k})$ and  $s_k = x_{k+1}-x_k$.

\item[S4] \textit{Sample size update.} Chose $ N_{k+1} \in \mathbb{N} $.

\item[S5] \textit{Spectral coefficient update.} Calculate $y_k = g_{\mathcal{N}_{k}}(x_{k+1})- \bar{g}_k$ where $g_{\mathcal{N}_{k}}(x_{k+1}) \in \partial f_{\mathcal{N}_{k}}(x_{k+1})$.
Set 
$\zeta_{k+1} = \min \lbrace\overline{\zeta},\max\lbrace\underline{\zeta}, \frac{s_k^T s_k}{s_k^T y_k}\rbrace\rbrace.$  
 
\item[S6] Set $k:=k+1$ and go to S1. 
\end{itemize}

{\bf{Remark.}} LS-SPS algorithm falls into the framework of SPS algorithm as $\alpha_k$ satisfies the condition \eqref{Armijo}. Thus, the whole convergence analysis presented for the SPS algorithm also holds for LS-SPS.

\section{Numerical results}
We performed preliminary numerical experiments on the set of binary classification problems listed in Table \ref{tabela_dataset}.  The problems are modeled by the $L_2$-regularized Hinge Loss. More precisely, we consider the following optimization problem for learning with a Support Vector Machine introduced in  \cite{pegasos} 
$$
\min_{x \in \Omega} f (x) : = 10||x||^2 + \frac{1}{N}\sum_{i=1}^{N} \max\lbrace{0, 1 - z_i x^T w_i\rbrace},
$$
$$\Omega : = \lbrace{x \in \mathbb{R}^n \;: ||x||^2 \leq 0.1 \rbrace},$$
where $w_i \in \mathbb{R}^n $ are the input features and $z_i \in \lbrace{1,-1\rbrace}  $ are the corresponding labels. Thus, we have a convex problem with the compact feasible set easy to project on. Moreover, for these kind of problems it is possible to calculate the descent direction and we use the procedure proposed in \cite[Algorithm 2, p.~1155]{kineski} as a subroutine that provides the descent property \eqref{descent}. We employ this  subroutine in all the tested algorithms to ensure the fair comparison. 

\begin{table}[h!]
\begin{center}
 \begin{tabular}{||c l c c c c ||} 
 \hline
  & Data set & $N$ & $n$ & $N_{train}$ & $N_{test}$ \\ [0.5ex] 
 \hline\hline
 1 & SPLICE \cite{SPLiADL}& 3175& 60 &2540 &635\\ 
 \hline
 2 & MUSHROOMS \cite{MUSH}& 8124& 112 &6500 &1624\\ 
 \hline
 3 & ADULT9 \cite{SPLiADL}& 32561& 123 &  26049&6512\\
 \hline
 4 & MNIST(binary) \cite{MNIST} &70000 & 784 & 60000&10000\\ [0.5ex] 
 \hline
\end{tabular}
\caption{Properties of the datasets used in the experiments.}\label{tabela_dataset}
\end{center}
\end{table}

Our numerical study has several  goals. It is designed to investigate: \\
a) whether the variable sample size approach remains  beneficial in the nonsmooth environment with bounded full sample; \\
b) whether introducing the  line search pays off;  \\
c) whether the spectral coefficient improves the efficiency of the projected subgradient method. 

We set the experiments as follows. The main criterion for comparison of the methods will be the computational cost modeled  by $FEV$ (number of function evaluations). More precisely, $FEV^m_k$ represents the number of scalar products needed for  method $m$ to calculate $x_k$ (starting from $x_0$). We also track the value of the true objective function across the iterations to observe the progress of the considered method. 

To answer the question a), we compare the VSS methods to their full sample counterparts. The extension $-F$ (e.g. SPS-F) will indicate that the full sample is used at every iteration, i.e., $N_k=N$ for all $k$. On the other hand, we assume the following sample size increase for the VSS methods: $N_{k+1} =\lceil\min\lbrace 1.1N_k,N \rbrace\rceil$ with $N_0=0.1 N$. Obviously there are many other choices which can be more efficient, but we choose this simple increase to be tested in the initial phase of the method evaluation. 
To address the question  b) we compare LS-SPS algorithm  to SPS algorithm with the standard choice of the  step size  $\alpha_k=1/k$.  Finally, to address c), we compare the proposed methods to the first order subgradient method denoted by LS-PS (Line Search Projected Subgradient), which can be viewed as a special case of LS-SPS with $ \underline{\zeta}= \overline{\zeta}=1.$ We also test the VSS and the full sample alternatives of the projected subgradient method: LS-PS and LS-PS-F, respectively. The results of the subgradient method with the choice of $\alpha_k=1/k$ were poor and thus not reported here.

The relevant parameters are as follows.   
The initial points are chosen randomly from $(0,1)$ interval and we use the same initial point for all the tested methods within one run. The step size parameters are $C_1=10^{-2}, C_2=10^2$  and the line search is performed with $\eta =10^{-4}$ and  $c=5$. For the proposed  spectral methods, the safeguard parameters are   $\underline{\zeta}=10^{-4}$ and $\overline{\zeta}=10^4$.   
 
We perform 5 independent runs  for each of the data sets which yields 20 runs of each method in total. A demonstrative run is presented in Figure \ref{MNIST_fig} where the objective function $f(x_k)$ is plotted against the $FEV_k$. It reveals that LS-SPS methods outperform other methods in a sense that reach tighter  vicinity of the solution. Even if we use the spectral coefficient, predetermined step size was not enough  to bring the sequence to the same vicinity as obtained by the LS counterparts. 
On the other hand, observing the LS-PS method which uses line search but without a spectral coefficient, we can see that the line search itself (without second order information) was not enough to push the subgradient method towards the solution. Finally, notice that the computational cost is reduced significantly by employing the VSS scheme in LS-SPS method. 

\begin{figure}[!ht]
     \includegraphics[scale=0.5,angle = 270]{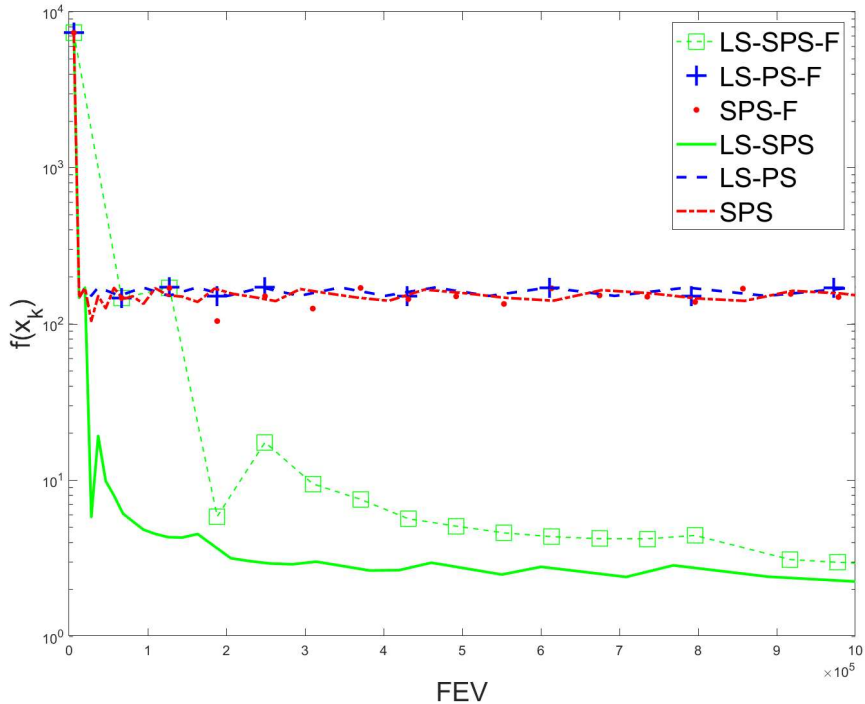}
       \caption{MNIST data set}\label{MNIST_fig}
\end{figure}

In order to compare the tested methods taking into account all the runs and data sets, we employ the metric based on the ideas of the performance profile \cite{pp}  adapted to the stochastic environment in \cite{andrea}. Roughly speaking, we estimate the probability of winning for each of the tested methods. For the considered run, the method $m$ wins if it reaches the vicinity $\tau$ of the solution with the smallest costs. Since the theoretical stopping criterion is non-existing, we stop the methods if the maximal number of scalar products is achieved. At each iteration $k$  we measure the distance from the solution by observing the relative error of method $m$ with respect to the optimal value, i.e., $r^m_k:=(f(x^m_k)-f^*)/f^*$. For each method $m$ and each run $l$ we register the first iteration $k(m,l)$ at which we have $r^m_{k(m,l)} \leq \tau$ and read the corresponding $FEV^m_{k(m,l)}$. Then, the method $m$ earns a point in run $l$ if $FEV^m_{k(m,l)}=\min_{j}  FEV^j_{k(j,l)}$. Finally, we estimate the probability of winning, denoted by $\pi $ by 
$$\pi=\frac{t}{T},$$
where $t$ is the number of earned points and $T$ is the total number of runs. 
Notice that in the described situation we can have more winners, in other words more methods can share the first place if they reached the goal with the same costs.

The results are presented in Figure \ref{Pobednik20_fig} for different  relative errors  $\tau \in [0.01, 3.5]$. They reveal that the VSS methods clearly outperform their full sample counterparts and that LS-SPS method turns out to be the best possible choice according to the  conducted experiments. The algorithms LS-PS and SPS reach 1 for very large values of the relative error $\tau$ which is not relevant, so we do not show this part of the graph. 

The Figure \ref{tau1_fig} represents classical performance profile (PP) graph for fixed  relative error  $\tau=1$. The FEV is kept as the criterion for the classical  PP as well. On the $y$-axes we plot the probability that the method is close enough to the best one, where "close enough" is determined by the value on $x$-axes denoted by $q$.  More precisely, retaining the same notation as above, the method $m$ earns a PP(q) point in run $l$ if $FEV^m_{k(m,l)}\leq q \min_{j}  FEV^j_{k(j,l)}$ and the plotted values correspond to  $$\pi_{PP}(q)=\frac{t_{PP}(q)}{T},$$
where $t_{PP}(q)$ is the number of earned PP(q) points for the considered method. Again, from this figure it is clear that LS-SPS is 
the most robust, i.e., it has the highest probability of being the optimal solver.

\begin{figure}[!ht]
        \includegraphics[scale=0.5,angle = 270]{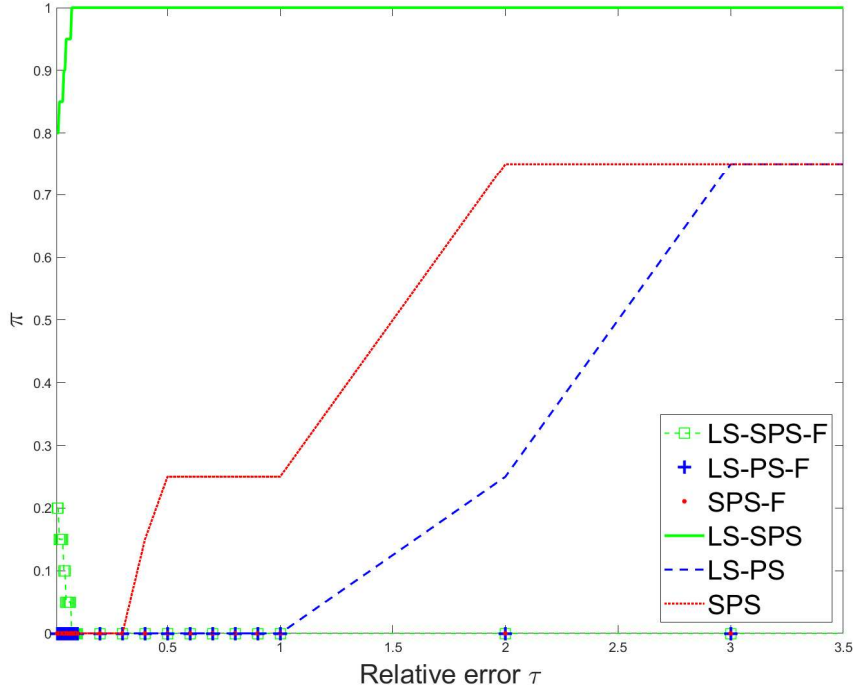}
              \caption{Empirical probabilities of winning ($\pi$) for different relative errors  ($\tau$).}\label{Pobednik20_fig}
\end{figure}

\begin{figure}[!ht]
        \includegraphics[scale=0.5,angle = 270]{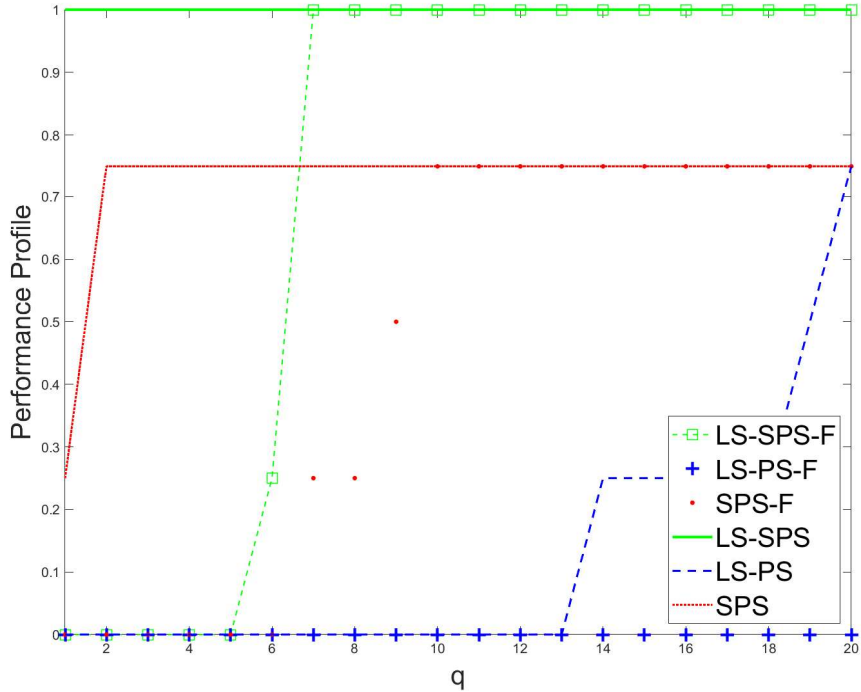}
              \caption{Performance profile for level of accuracy  $\tau=1$}\label{tau1_fig}
\end{figure}

\subsection{Additional comparison}

In this subsection we show additional numerical results in order to compare proposed algorithm with the proximal bundle method (PBM). It is known that PBM gives the best result under a fixed number of iterations. The reason behind this is the fact that the number of constraints in the quadratic program solved by PBM may grow linearly with the number of iterations \cite{senior}. Accordingly, PBM may become significantly slower when the number of iterations become larger. For that reasons, we compare fixed number of iterations of LS-SPS-F with PBM that use full sample size in all iterations, and after that LS-SPS with VSS PBM (where the sample is changed in the same way as in LS-SPS through iterations). 

In order to ensure the fair comparison, we choose several different combinations of initial parameters for PBM. Table \ref{tabela1} summarizes the properties of the observed methods, where $\gamma$ is proximity control parameter, $m$ is descent coefficient, $\epsilon$ is tolerance parameter and $\omega$ is decay coefficient.
Detailed information about these parameters and implementations in Matlab of PBM  are available at \cite{kod}.

The Figure \ref{fig_bundle1} shows the results for Full and VSS version of LS-SPS and PBM algorithms on MNIST data set, while the results on other three data sets are similar. The objective function $f(x_k)$ is plotted against the iteration $k$ and the y-axes are in logarithmic scale. The results show that the proposed algorithms (LS-SPS-F and LS-SPS) outperform the observed PBM counterparts.
\begin{table}[h!]
\begin{center}
 \begin{tabular}{||c c c c c c ||} 
 \hline
  PBM &  $1$ & $2$ & $3$ & $4$ & $5$\\ [0.5ex] 
 \hline\hline
$\gamma$ &	1	&1	&1&	1	&1
\\ 
 \hline
 $m$ & 0.01	&0.1&	0.01&	0.01&	0.01
\\ 
 \hline
 $\epsilon$&	0.1&	0.1&	0.01&	0.1&	0.1
\\
 \hline
 $\omega$ &	0.5	&0.5	&0.5&	0.9&	0.1
\\ [0.5ex] 
 \hline
\end{tabular}
\caption{The initial parameters for PBM.}\label{tabela1}
\end{center}
\end{table}

\begin{figure}[htbp]
    \hspace*{-0.3in}
  \includegraphics[width=0.45\textwidth,angle = 270]{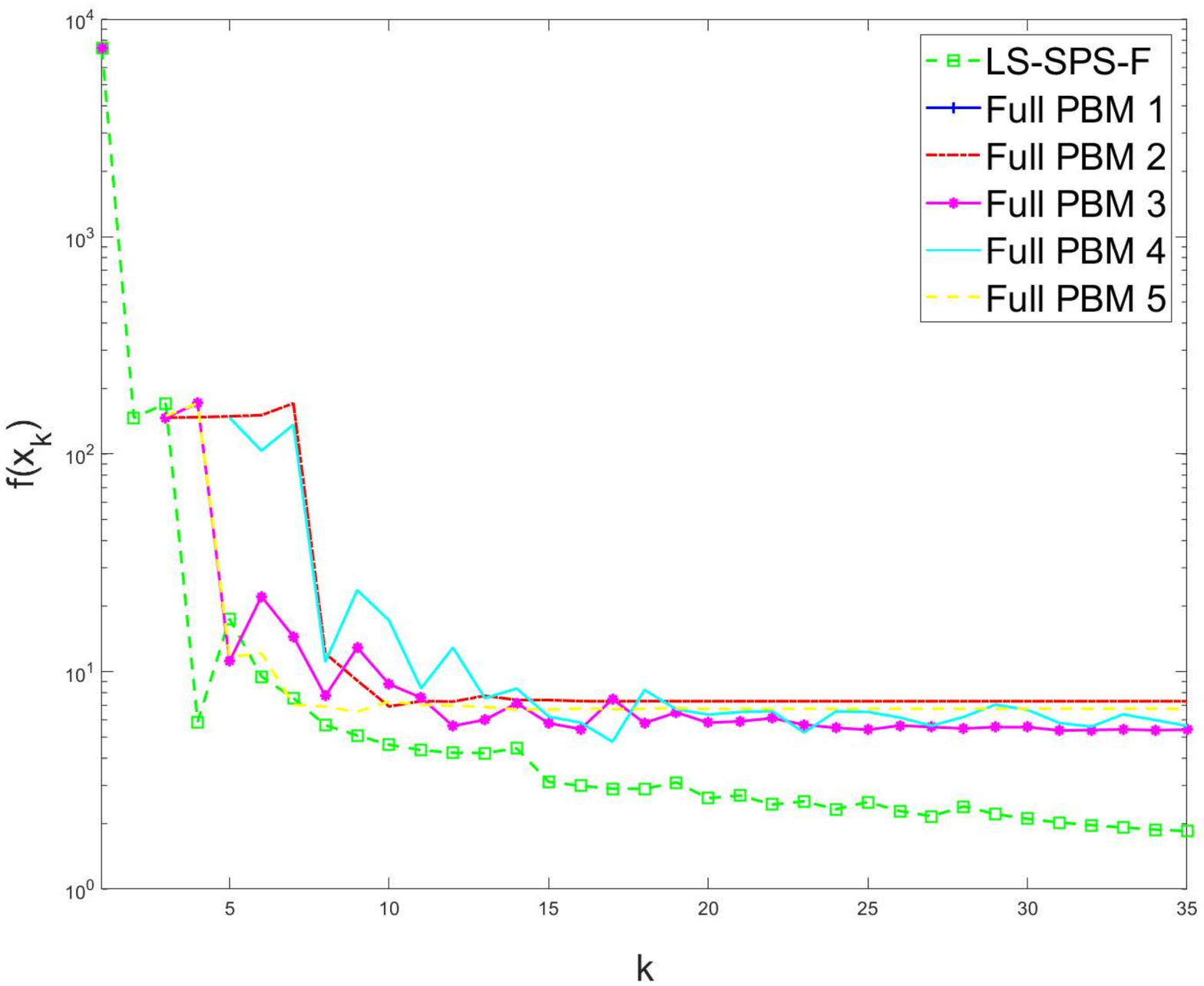} \hspace*{-0.3in}
   \includegraphics[width=0.45\textwidth,angle = 270]{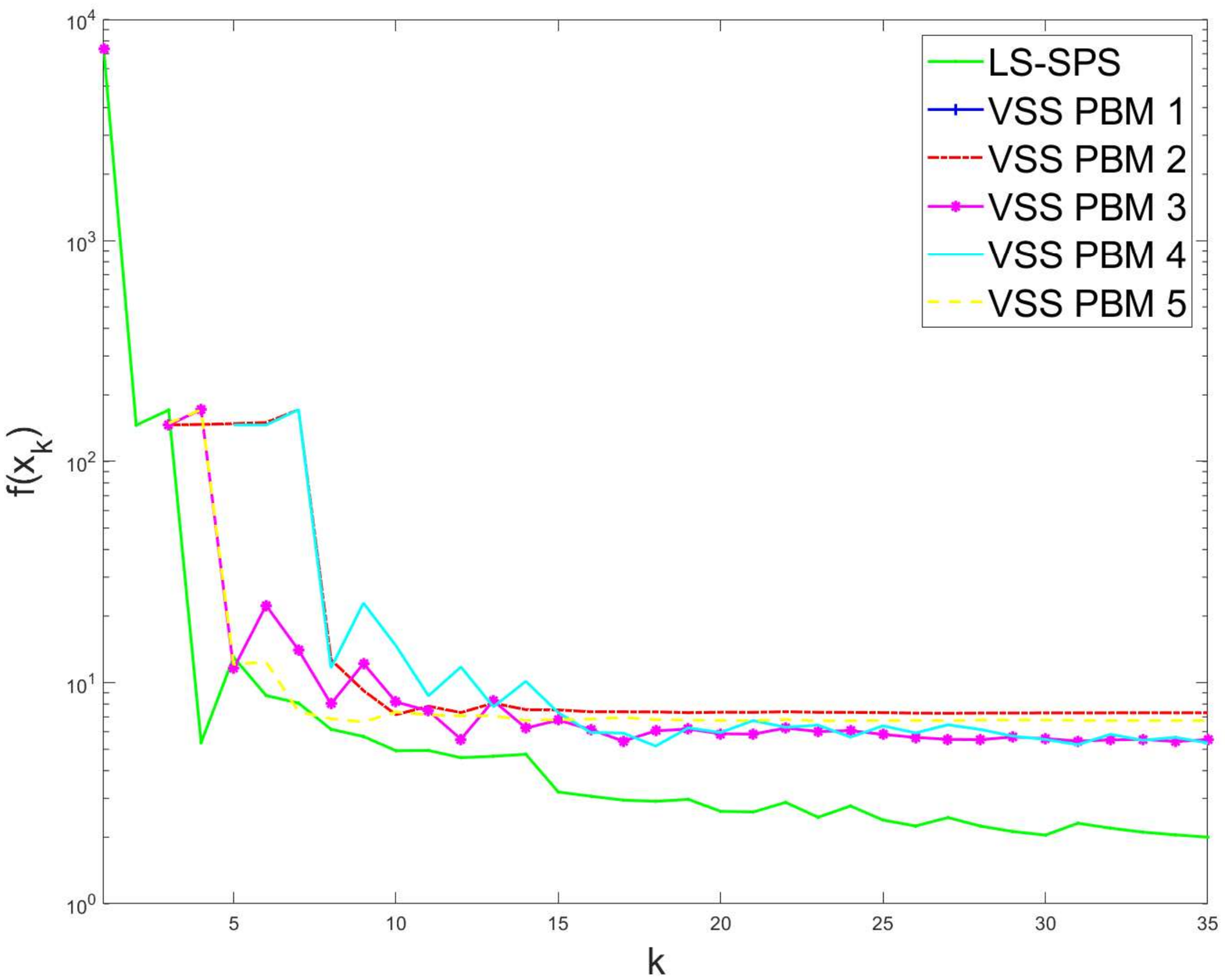}
   \\
   \caption{LS-SPS-F against Full PBM (left) and  LS-SPS against VSS PBM (right). MNIST data set.}
   \label{fig_bundle1}
\end{figure}

\section{Conclusions}

The  method SPS for solving constrained optimization problems with a nonsmooth objective function in the form of mathematical expectation is proposed. We assume that the feasible set is easy to project on and employ orthogonal projections in order to ensure feasibility of the iterates. The  SAA  is used to estimate the objective function and Variable Sample Size strategy is applied.  SPS combines an SAA subgradient with the spectral coefficient in order to provide a suitable direction.  The step sizes are chosen from the predefined interval, while the choice is further specified in the line search version - LS-SPS. The almost sure convergence of the method is proved under the standard assumptions in stochastic environment. We provide convergence analysis both for bounded and unbounded sample sizes. In the later case we show that the sequence of iterates is bounded if the sample size growth is fast enough. We also analyze the important special case - finite sum problems, for which we prove deterministic convergence under the reduced set of assumptions. Numerical experiments are conducted on the set of machine learning problems modeled by the Hinge Loss binary classification. The methods are compared through performance profile type of graphs where the main criterion is the computational cost modeled by the number of scalar products. The  initial results yield several conclusions, including the one that VSS outperforms the full sample as expected. We also conclude that the spectral coefficient is beneficial, but it achieves the best performance when combined with the Armijo-like line search procedure. Additionally, numerical results show clear advantages of the proposed LS-SPS method with respect to  Proximal Bundle Method. Since the current results are promising, the future work should include adaptive VSS schemes, inexact projections and further investigation of the step size selection and the spectral coefficient in stochastic environment. 

\medskip
 \noindent{\bf Acknowledgement.} We are grateful  to the associate editor and the anonymous referee for their comments that helped us to improve the paper. 
\\
{\bf Funding} This work is supported by the Ministry of Education, Science and Technological Development, Republic of Serbia.
\\
{\bf Availability statement} The datasets analysed during the current study are available in the the MNIST database
of handwritten digits \cite{MNIST},  LIBSVM Data: Classification (Binary Class) \cite{SPLiADL} and UCI Machine Learning Repository \cite{MUSH}.

\section*{Declarations} 
{\bf Conflict of interest} The authors declare no competing interests.

\end{document}